\documentclass[12pt]{article}
\newcommand{\editor}{Sudipta Mallik}
\newcommand{\ArticleType}{Research Note} 
\newcommand{\RecievedDate}{Mar 27, 2024}
\newcommand{\RevisedDate}{Nov 16, 2024}
\newcommand{\AcceptedDate}{Nov 18, 2024}
\newcommand{\PublishedDate}{Nov 20, 2024}
\newcommand{\JournalIndex}{Volume 3 (2024), Pages 35--43}
\newcommand{\LastName}{Bapat}
\newcommand{\JournalIndexShort}{\LastName\, et al./ American Journal of Combinatorics 3 (2024) 35--43}
\usepackage{fancyhdr}
\usepackage{amsmath,amsthm,amssymb,mathtools,url,graphicx,pdfpages,tikz,rotating}
\usepackage[bookmarks=true,pdfborder={0 0 0}]{hyperref}
\usepackage[affil-it]{authblk}
\usepackage[left=1in,right=1in,top=1.2in, bottom=1in]{geometry}
\usepackage{setspace}
\usepackage{geometry}
\usepackage[english]{babel}
\usepackage[utf8x]{inputenc}
\usepackage{indentfirst}
\usepackage{xcolor}

\newtheorem{theorem}{Theorem}[section]
\newtheorem{lemma}[theorem]{Lemma}
\newtheorem{corollary}[theorem]{Corollary}
\theoremstyle{definition}

\newtheorem{example}[theorem]{Example}

\newtheorem{proposition}[theorem]{Proposition}

\numberwithin{equation}{section}

\DeclareMathOperator{\per}{per}
\DeclareMathOperator{\sign}{sgn}

\begin{document}
\setcounter{page}{35}
\noindent {\color{teal}\bf\large American Journal of Combinatorics} \hfill \ArticleType\\
\JournalIndex

\title{Computing the permanental polynomial of $4k$-intercyclic bipartite graphs}

\author{Ravindra B. Bapat, Ranveer Singh, and Hitesh Wankhede*}

\affil{\normalsize\rm (Communicated by \editor)
\vspace*{-24pt}}
\date{}

{\let\newpage\relax\maketitle}

\begin{abstract}
Let $G$ be a bipartite graph with adjacency matrix $A(G)$. The characteristic polynomial $\phi(G,x)=\det(xI-A(G))$ and the permanental polynomial $\pi(G,x) = \per(xI-A(G))$ are both graph invariants used to distinguish graphs. For bipartite graphs, we define the modified characteristic polynomial, which is obtained by changing the signs of some of the coefficients of $\phi(G,x)$. For $4k$-intercyclic bipartite graphs, i.e., those for which the removal of any $4k$-cycle results in a $C_{4k}$-free graph, we provide an expression for $\pi(G,x)$ in terms of the modified characteristic polynomial of the graph and its subgraphs. Our approach is purely combinatorial in contrast to the Pfaffian orientation method found in the literature to compute the permanental polynomial.

\end{abstract}

\renewcommand{\thefootnote}{\fnsymbol{footnote}} 
\footnotetext{\hspace*{-22pt} * Corresponding author
Received \RecievedDate; Revised \RevisedDate; Accepted \AcceptedDate; Published \PublishedDate \\
\copyright\, The author(s). Released under the CC BY 4.0 International License}
\renewcommand{\thefootnote}{\arabic{footnote}}

\section{Introduction and preliminaries}

We consider simple and undirected graphs. Let $G$ be a graph with the vertex set $V(G)=\{v_1,v_2,\dots,v_n\}$. The \emph{adjacency matrix} $A(G)=(a_{i,j})$ of a graph $G$ is defined such that $a_{i,j} = 1$ if $v_i$ and $v_j$ are adjacent and $0$ otherwise where $i, j \in \{1,2,\dots, n\}$. The \emph{determinant} and the \emph{permanent} of $A(G)$, are defined as 
\begin{equation*}
    \det (A(G))=\sum_{\sigma \in S_n}\sign(\sigma)\prod_{i=1}^{n}a_{i,\sigma(i)} \text{ and }\per(A(G))=\sum_{\sigma \in S_n}\prod_{i=1}^{n}a_{i,\sigma(i)},
\end{equation*}
respectively, where $S_n$ is the set of all permutations of the set $\{1,2,\dots, n\}$ and $\sign(\sigma)$ is the signature of the permutation $\sigma$. While the determinant can be computed in polynomial time using the Gaussian elimination method, and the fastest known algorithm runs in $\mathcal{O}(n^{2.371552})$ time \cite{aho1974design,williams2024new}, computing the permanent is notoriously difficult, as it is $\#$P-complete \cite{valiant1979complexity}. The ``Permanent vs. Determinant Problem'' about symbolic matrices in computational complexity theory is as follows: ``Can we express the permanent of a matrix as the determinant of a (possibly polynomially larger) matrix?'' For an upper bound on the size of the larger matrix, see \cite{grenet2011upper}, and for a survey on lower bounds, see \cite{agrawal2006determinant}. The problem  ``Given a $(0,1)$-matrix $A$, under what conditions, changing the sign of some the nonzero entries yields a matrix $B$ such that $\per(A)=\det(B)$?'' is famously known as ``Polya Permanent Problem,'' \cite{polya_aufgabe_1913} and it is equivalent to twenty-three other problems listed in \cite{mccuaig_polyas_2004}. Immanants are matrix functions that generalize determinant and permanent, and their complexity dichotomy was also recently studied \cite{curticapean2021full}. 

The \emph{characteristic polynomial} and the \emph{permanental polynomial} of graph $G$ on $n$ vertices are univariate polynomials defined as 
\begin{equation*}
    \phi(G,x)=\det(xI-A(G))\text{ and } \pi(G,x)=\per(xI-A(G)),
\end{equation*}
respectively, where $I$ is the identity matrix of order $n$. The characteristic and the permanental polynomials are graph invariants, and they can be used to distinguish graphs towards Graph Isomorphism Problem \cite{vandam_haemers}. But the permanental polynomial is not studied in great detail as compared to the characteristic polynomial, probably due to its computational difficulty. Some computational evidence suggests that the permanental polynomial is better than the characteristic polynomial while distinguishing graphs \cite{dehmer2017highly,liu2014enumeration}. We are interested in finding ways to compute the permanental polynomial efficiently; one way to do that is by expressing the permanental polynomial in terms of the characteristic polynomial. For an excellent survey on the permanental polynomial, we refer to \cite{2016permanental}. 

Let
\begin{equation*}
    \phi(G,x)=\sum_{i=0}^{n} a_ix^{n-i}\text{ and } \pi(G,x)=\sum_{i=0}^{n} b_ix^{n-i}.
\end{equation*}
The interpretation of these coefficients is given as
\begin{align}\label{eqn:sachs_interp}
a_i = \sum_{U_i} (-1)^{p(U_i)} 2^{c(U_i)} \text{ \cite{sachs1964beziehungen} }, \;\;
b_i = (-1)^i\sum_{U_i} 2^{c(U_i)} \text{ \cite{merris1981permanental} },
\end{align}
where the summation is taken over all the Sachs subgraphs $U_i$ (subgraphs whose components are either cycles or edges) of $G$ on $i$ vertices, $p(U_i)$ denotes the number of components in $U_i$, and $c(U_i)$ denotes the number of components in $U_i$ that are cycles. It is easy to see the following.
\begin{proposition}\label{prop:bip_coef}\cite{borowiecki1985spectrum, cvetkovic1979spectra}
A graph $G$ is bipartite if and only if $a_i=b_i=0$ for each odd $i$.
\end{proposition}
Hence, for a bipartite graph $G$, we have
\begin{equation}\label{eqn:bip_exp}
    \phi(G,x)=\sum_{i=0,2,4,\dots} a_{i}x^{n-i}\text{ and } \pi(G,x)=\sum_{i=0,2,4, \dots} b_{i}x^{n-i}. 
\end{equation}
Defining $f_i= b_i - (-1)^{i/2}a_i$ for each nonnegative even integer $i$, we introduce a \textit{modified characteristic polynomial} and also a new graph polynomial for bipartite graphs as 
\begin{equation*}
    \phi_p(G,x)=\sum_{i=0,2,4,\dots} (-1)^{i/2}a_ix^{n-i}\text{ and } f(G,x)=\sum_{i=0,2,4,\dots} f_ix^{n-i},
\end{equation*}
respectively, such that we have 
\begin{equation}\label{eqn:pi=phi+f}
    \pi(G,x) = \phi_p(G,x) + f(G,x). 
\end{equation} 

We denote a cycle of length $k$ by $C_{k}$. For each positive integer $k$, we refer to a $C_{4k}$ by a $4k$\textit{-cycle}. A graph is called $C_{4k}$-\textit{free} if it does not contain a $4k$-cycle for all positive integers $k$. A graph is called $4k$-\textit{intercyclic} if it does not contain two vertex-disjoint $4k$-cycles (these $4k$-cycles could be of different length). Equivalently, after removal of the vertices of any $4k$-cycle from a $4k$-intercyclic graph, the resultant graph is $C_{4k}$-free (see Figure \ref{fig:pap_ex5} for an example). 
\begin{figure}[h!]
     \centering
     \includegraphics[width=0.4\textwidth]{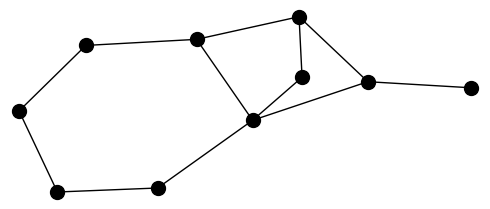}
     \caption{Example of a $4k$-intercyclic bipartite graph.}
     \label{fig:pap_ex5}
 \end{figure}

In 1985, Borowiecki proved the following. 
\begin{theorem}\cite{borowiecki1985spectrum}
    Let $G$ be a bipartite graph with the spectrum $\{\lambda_1,\lambda_2,\dots, \lambda_n\}$. Then, $G$ is $C_{4k}$-free if and only if its per-spectrum is $\{i\lambda_1,i\lambda_2,\dots, i\lambda_n\}$.\footnote{The \textit{spectrum} and the \textit{per-spectrum} of a graph $G$ are the multiset of all roots of its characteristic polynomial and the permanental polynomial, respectively, and $i$ is an imaginary unit.}
\end{theorem}

By inspecting the proof of this theorem, we notice that a bipartite graph $G$ is $C_{4k}$-free if and only if $\pi(G,x) = \phi_p(G,x)$ (see Corollary \ref{cor:4kbip}). As a result, the permanental polynomial of $C_{4k}$-free bipartite graphs can be computed directly through the modified characteristic polynomial. Yan and Zhang, in 2004, found that the permanental polynomial of a larger class of bipartite graphs can be computed using the characteristic polynomial of some oriented graph. They proved the following.

\begin{theorem}\cite{yan2004permanental}\label{thm:yan_zhang}
    Let $G$ be a bipartite graph with $n$ vertices containing no subgraph that is an even subdivision of $K_{2,3}$. Then there exists an orientation $G^e$ of $G$ such that  
    $\pi(G, x) = \det(xI - A(G^e)),$ where $A(G^e)$ denotes the skew adjacency matrix of $G^e$. 
\end{theorem}

Later Zhang and Li, in 2012, proved the converse of this statement. 
\begin{theorem}\cite{zhang2012computing}\label{thm:zhang_li}
    There exists an orientation $G^e$ of a bipartite graph $G$ such that $\pi(G, x) = \det(xI - A(G^e))$ if and only if $G$ contains no even subdivision of $K_{2,3}$. 
\end{theorem}

For the definition of an even subdivision of a graph, see \cite{yan2004permanental,zhang2012computing}. Zhang and Li also show that bipartite graphs that do not contain an even subdivision of $K_{2,3}$ are planar and admit Pfaffian orientation. They also give characterization and recognition of such graphs, which leads to a polynomial time algorithm to compute the permanental polynomial of such bipartite graphs. Next, we reformulate Theorem \ref{thm:zhang_li}.
\begin{theorem}\label{thm:zhangli_reform}
    There exists an orientation $G^e$ of a graph $G$ such that $\pi(G, x) = \det(xI - A(G^e))$ if and only if $G$ is a bipartite graph containing no even subdivision of $K_{2,3}$. 
\end{theorem}
\begin{proof}
	Suppose that there is an orientation $G^{e}$ such that $\pi(G, x) = \det(xI - A(G^e))$ holds. It is enough to show that $G$ is bipartite. The conclusion of the theorem then follows by Theorem \ref{thm:zhang_li}. Since the skew-adjacency matrix $A(G^e)$ is skew-symmetric, it has purely
	imaginary eigenvalues. Hence, the permanental polynomial can be expressed as $\pi(G,x)= (x-i\lambda_1)(x-i\lambda_2)\dots(x-i\lambda_n)$ for some real numbers $\lambda_1,\lambda_2,\dots, \lambda_n$. When $k$ is odd, the imaginary unit $i$ is a factor of the coefficient $b_k$ of $x^{n-k}$. Since the coefficients of $\pi(G,x)$ must be real, it follows that $b_k=0$ for all odd $k$, and by Proposition \ref{prop:bip_coef}, $G$ is bipartite.
\end{proof}
Theorem \ref{thm:zhangli_reform} suggests that the orientation approach in computing the permanental polynomial only works for the class of bipartite graphs that do not contain an even subdivision of $K_{2,3}$. In this article, we give a formula to compute $\pi(G,x)$ for the class of $4k$-intercyclic bipartite graphs (a superset of the class of $C_{4k}$-free bipartite graphs). This is done by expressing $f(G,x)$ in terms of the modified characteristic polynomial of the subgraphs of $G$. Our approach is combinatorial rather than based on Pfaffian orientation. Note that the class of $4k$-intercyclic bipartite graphs is different from and not a subset of the class of bipartite graphs that do not contain an even subdivision of $K_{2,3}$.

We would also like to mention that our result seem to be in the same spirit as Polya's scheme completed by Galluccio and Loebl in 1999. They proved that the generating function of the perfect matchings of a graph of genus $g$ may be written as a linear combination of $4^g$ Pfaffians, and as a consequence obtained the following result. 
\begin{theorem}\cite{galluccio1999theory}\label{thm:gallu_loeb}
Let $A$ be a square matrix. Then $\per(A)$ may be expressed as a linear combination of terms of the form $\det(A^i)$, $i = 1,\dots, 4^g$, where each $A^i$ is obtained from $A$ by changing the sign of some entries and $g$ is the genus of the bipartite graph corresponding to the biadjacency matrix $A$.
\end{theorem}

\section{Main result and application}
\begin{theorem}\label{thm:boro_gen}
    Let $G$ be a $4k$-intercyclic bipartite graph. Then, 
    \begin{equation*}
        \pi(G,x) = \phi_p(G,x) + 4 \sum_{R\in \mathcal{C}_{4k}(G)} \phi_p(G\backslash R,x), 
    \end{equation*}
    where $\mathcal{C}_{4k}(G)$ denotes the set of all $4k$-cycles in $G$.
 \end{theorem}
 To prove this theorem, we need the following lemma.
\begin{lemma}\label{lem:fgx}
	Let $G$ be a bipartite graph. Then, for each nonnegative even integer $i$, we have
    \begin{align*}
    f_i=\sum_{j=1,3,5,\ldots}2^{j+1}\sum_{\substack{U_i \text{} \text{ containing} \\ \text{exactly $j$ } 4k\text{-cycles}}} 2^{t(U_i)},
    \end{align*}
    where $U_i$ denotes a Sachs subgraph on $i$ vertices, and $t(U_i)$ is the number of $(4k+2)$-cycles in $U_i$.
\end{lemma}
\begin{proof}[Proof of Lemma \ref{lem:fgx}]
	In a bipartite graph, there can be two types of cycles: $4k$-cycles or $(4k+2)$-cycles. Let $U_i$ be any Sachs subgraph on $i$ vertices, and $s(U_i)$ and $t(U_i)$ be the number of $4k$-cycles and $(4k+2)$-cycles in it respectively. Then, $U_i$ can be expressed as
    $$ U_i = \{C_{4k_1} \cup \dots \cup C_{4k_{s(U_i)}} \} \cup \{ C_{4l_1+2} \cup \dots \cup  C_{4l_{t(U_i)}+2} \}\cup \underbrace{\{K_2 \cup \dots \cup K_2 \}}_{r(U_i)- \text{times}},$$
    such that $p(U_i)=s(U_i)+t(U_i)+r(U_i)$, $c(U_i)=s(U_i)+t(U_i)$ and $$i = 4(k_1+\dots +k_{s(U_i)}) + 4(l_1+\dots +l_{t(U_i)}) + 2 (t(U_i) + r(U_i)).$$ Check that $s(U_i)+t(U_i)+r(U_i) \equiv i/2 + s(U_i) $ (mod $2$). Using this fact, the coefficients of the characteristic polynomial and the permanental polynomial given in Equation \ref{eqn:sachs_interp} can be written as 
    \begin{equation*}
        (-1)^{i/2} a_i = \sum_{U_i}(-1)^{s(U_i)} 2^{s(U_i)+t(U_i)}\text{ and }
        b_i = \sum_{U_i} 2^{s(U_i)+t(U_i)},
    \end{equation*}
    respectively ($(-1)^i=1$ since $i$ is even). Since $f_i=b_i-(-1)^{i/2}a_i$, we get
    \begin{gather}\label{eqn:fi_forthm}
    f_i =\sum_{U_i} \left ( 1- (-1)^{s(U_i)} \right ) 2^{s(U_i)+t(U_i)} 
    =\sum_{\substack{U_i \text{} \text{ containing an odd } \\\text{number of } 4k\text{-cycles}}} 2^{s(U_i)+1} 2^{t(U_i)}  \\
    =\sum_{j=1,3,5,\ldots}2^{j+1}\sum_{\substack{U_i \text{} \text{ containing} \\ \text{exactly $j$ } 4k\text{-cycles}}} 2^{t(U_i)}.\nonumber
    \end{gather}
Note that the contribution in Equation \ref{eqn:fi_forthm} of the Sachs subgraphs in which we have exactly an even number of $4k$-cycles vanishes. 
\end{proof}
\begin{proof}[Proof of Theorem \ref{thm:boro_gen}]
    Since $G$ is $4k$-intercyclic, the subgraph $G\backslash R$ is $C_{4k}$-free for any $R \in \mathcal{C}_{4k}(G)$. Similarly, any Sachs subgraph $U_i$ of $G$ can contain at most one $4k$-cycle, that is, $s(U_i)\leq 1$. Using Lemma \ref{lem:fgx}, we have \begin{align*}
    f_i=4\sum_{\substack{U_i \text{} \text{ containing} \\ \text{exactly one } 4k\text{-cycle}}} 2^{t(U_i)} = 4
    \sum_{\substack{R \in \mathcal{C}_{4k}(G)}} \sum_{\substack{U_i \\\text{} \text{ containing } R}} 2^{t(U_i)}
    \end{align*}
    for each nonnegative even integer $i$. For any fixed $R \in \mathcal{C}_{4k}(G)$, there is a one-to-one correspondence between the Sachs subgraphs in $G$ containing $R$ and the Sachs subgraphs in $G\backslash R$. To see this, consider a Sachs subgraph $U_i$ in $G$ containing $R$. Let $l_R=i-|V(R)|$ and $W_{l_R} = U_i \backslash R$ be the unique subgraph of $G\backslash R$ corresponding to $U_i$. Since $G$ is $4k$-intercyclic, removing $R$ ensures that $W_{l_R}$ is a Sachs subgraph of $G\backslash R$ since both $G\backslash R$ and $U_i \backslash R$ do not contain $4k$-cycles. Conversely, adding $R$ back to a Sachs subgraph $W_{l_R}$ in $G \backslash R$ uniquely reconstructs a corresponding subgraph $U_i$ in $G$. This establishes a one-to-one correspondence. Moreover, the number of cycles in $W_{l_R}$ is the same as the number $(4k+2)$-cycles in $U_i$, i.e., $c(W_{l_R})=t(U_i)$. As a result, we have
    \begin{align*}
    f_i=4
    \sum_{\substack{R \in \mathcal{C}_{4k}(G)}} \sum_{W_{l_R}} 2^{c(W_{l_R})}.
    \end{align*}
    Now consider the polynomial
    \begin{gather*}
        f(G,x)  = \sum_{i=0,2,4,\dots} f_ix^{n-i} 
        =4\sum_{i=0,2,4,\dots} 
        \sum_{\substack{R \in \mathcal{C}_{4k}(G)}} \sum_{W_{l_R}} 2^{c(W_{l_R})}x^{n-i} \\
        =
        4 \sum_{\substack{R \in \mathcal{C}_{4k}(G)}}  \sum_{l_R=0,2,4,\dots}\sum_{W_{l_R}} 2^{c(W_{l_R})} x^{(n-|V(R)|)-l_R}
        \\=4 \sum_{\substack{R \in \mathcal{C}_{4k}(G)}} \pi(G\backslash R,x ),
    \end{gather*}
    where the second last step follows from the rearrangement of sums, and the last step by Equation \ref{eqn:sachs_interp}, \ref{eqn:bip_exp} and the fact that $G\backslash R$ is a $C_{4k}$-free bipartite graph on $n-|V(R)|$ vertices.
    By Equation \ref{eqn:pi=phi+f}, we conclude 
    \begin{equation*}
        \pi(G,x) = \phi_p(G,x) + 4 \sum_{R\in \mathcal{C}_{4k}(G)} \pi(G\backslash R,x).
    \end{equation*}
    Since $G\backslash R$ is $C_{4k}$-free, the application of this expression to it leads to $\pi(G\backslash R,x) = \phi_p(G\backslash R,x)$ proving the theorem. 
\end{proof}

 \begin{example}
    Consider the $4k$-intercyclic graph $G$ shown in Figure \ref{fig:pap_ex5}. It contains three $4$-cycles and two $8$-cycles, and removal of each of them from the graph yields the following subgraphs: $P_5 \cup K_1$, $P_4 \cup K_2$, $P_4\cup K_1\cup K_1$, $K_2$ and $K_1 \cup K_1$, respectively. Then, using Theorem \ref{thm:boro_gen},
    \begin{gather*}
        \pi(G,x)=\phi_p(G,x) + 4( \phi_p(P_5 \cup K_1,x)+ \phi_p(P_4 \cup K_2,x) + \phi_p(P_4\cup K_1 \cup K_1,x)\\ + \phi_p(K_2,x) + \phi_p(K_1 \cup K_1,x) ).
    \end{gather*}
    We need to do the following computations: $\phi_p(G,x)=x^{10}+12x^8+40x^6+47x^4+18x^2$ + 1, $\phi_p(P_5 \cup K_1,x)=x^6+4x^4+3x^2$, $\phi_p(P_4 \cup K_2,x)=x^6+4x^4+4x^2+1$, $\phi_p(P_4\cup K_1\cup K_1,x)=x^6+3x^4+x^2$, $\phi_p(K_2)=x^2+1$, $\phi_p(K_1 \cup K_1,x) = x^2$. Hence, we get $\pi(G,x)=x^{10}+12x^8+52x^6+91x^4+58x^2+9.$
    Note that Theorem \ref{thm:yan_zhang} and \ref{thm:zhang_li} are not applicable for this graph $G$ as it contains $K_{2,3}$. 
 \end{example}

 The following corollary shows that Theorem \ref{thm:boro_gen} is a generalization of Borowiecki's proof idea for computational purposes at least. 
 
 \begin{corollary}\label{cor:4kbip}\cite{borowiecki1985spectrum}
     A bipartite graph $G$ is $C_{4k} \text{-free}$ if and only if $\pi(G,x)=\phi_p(G,x)$.
 \end{corollary}
 \begin{proof}
    The forward implication easily follows from Theorem  \ref{thm:boro_gen}. Suppose $\pi(G,x)=\phi_p(G,x)$ holds, then from Equation \ref{eqn:fi_forthm}, we have
    \begin{align*}
    f_i=\sum_{\substack{U_i \text{} \text{ with an odd }\\\text{number of } 4k\text{-cycles}}} 2^{s(U_i)+1} 2^{t(U_i)}=0,\end{align*}
    for each $i$. Suppose, on the contrary, that $G$ contains a $4k$-cycle for some $k$. Then, there exists a Sachs subgraph $U_{4k}=C_{4k}$, and it contains an odd number of $4k$-cycles. Hence, $f_{4k}\neq 0$, and we get a contradiction which concludes that $G$ is $C_{4k} \text{-free}$. 
 \end{proof}

By Theorem \ref{thm:boro_gen}, the computation of permanental polynomial of a $4k$-intercyclic bipartite graph requires listing all the $4k$-cycles in it. For any graph on $n$ vertices, all cycles of length up to $\log n$ can be found in polynomial time using the color coding method of Alon, Yuster and Zwick \cite{alon1995color}. By an algorithm of Birmelé et al. \cite{birmele2013optimal}, listing all cycles in a graph requires $\mathcal{O}(m + \sum_{c\in \mathcal{C}(G)}|c|)$ time, where $m$ is the number of edges and $\mathcal{C}(G)$ is the set of all cycles. Hence, the permanental polynomial of a $4k$-intercyclic bipartite graph can be computed in polynomial time if the length of the largest $4k$ cycle is bounded by $\log n$ or if the number of cycles is bounded by a polynomial in $n$.

We now discuss an application of Theorem \ref{thm:boro_gen} in constructing cospectral graphs. Recall that two graphs $G_1$ and $G_2$ are said to be \textit{cospectral} if they have the same characteristic polynomial, that is, $\phi(G_1,x)=\phi(G_2,x)$. Similarly, we say that they are \textit{per-cospectral} if they have the same permanental polynomial, that is, $\pi(G_1,x)=\pi(G_2,x)$. Since $\phi_p(G,x)$ can be recovered from $\phi(G,x)$, it follows from Corollary \ref{cor:4kbip} that two $C_{4k}$-free bipartite graphs $G_1$ and $G_2$ are cospectral if and only if they are per-cospectral \cite{borowiecki1985spectrum}. Next, we give a general procedure to construct $4k$-intercyclic bipartite graphs that are simultaneously cospectral as well as per-cospectral. Let $p$ be some polynomial. Then, corresponding to $p$, we define a class of $4k$-intercyclic bipartite graphs $\mathcal{G}_p = \{G \;|\; f(G,x)=p\}$ . 
\begin{theorem}\label{thm:gen}
	Let $G_1, G_2 \in \mathcal{G}_p$ for some polynomial $p$. Then, $G_1$ and $G_2$ are cospectral if and only if they are per-cospectral.  
\end{theorem}
\begin{proof}
	Write $\pi(G_1,x) = \phi_p(G_1,x) + f(G_1,x)$, and $\pi(G_2,x) = \phi_p(G_2,x) + f(G_2,x)$. But since $G_1, G_2 \in \mathcal{G}_p$, we have $f(G_1,x) = f(G_2,x)$. Hence, $\pi(G_1,x) - \pi(G_2,x) = \phi_p(G_1,x) - \phi_p(G_2,x)$. By definition, $\phi_p(G_1,x) = \phi_p(G_2,x)$ if and only if $\phi(G_1,x) = \phi(G_2,x)$. Then, it follows that $\phi(G_1, x)=\phi(G_2, x)$ if and only if $\pi(G_1, x)=\pi(G_2, x)$.  
\end{proof}
Three such classes of $4k$-intercyclic bipartite graphs are given below. 
\begin{enumerate}
    \item \textit{The class of all $C_{4k}$-free bipartite graphs}. Let $p=0$, then $\mathcal{G}_0 = \{G \; | \; f(G,x)=0\}$. By Corollary \ref{cor:4kbip}, every graph in this class is $C_{4k}$-free.  
    \item \textit{The class of all bipartite graphs with exactly $l$ $C_4$'s and no other cycle such that $G \backslash C$ is an edgeless graph for any cycle $C$, where $l$ is some positive integer}. Let $p=4lx^{n-4}$ for a given $n$ and $l$, then $\mathcal{G}_{4lx^{n-4}}= \{G \; | \; f(G,x)=4lx^{n-4}\}$. By Theorem \ref{thm:boro_gen}, we have $\sum_{R\in \mathcal{C}_{4k}(G)} \phi_p(G\backslash R,x) = lx^{n-4}$. Now the degree $x^{n-4}$ is achieved in $\phi_p(G\backslash R,x)$ only when $R$ is $C_4$. Hence, there are $l$ $C_4$'s in $G$ such that $G \backslash C$ is an edgeless graph for any cycle $C$. 
    \item \textit{The class of all unicyclic bipartite graphs with a $C_4$ such that any two graphs in this class are cospectral after the removal of $C_4$}. This follows easily from Theorem \ref{thm:boro_gen}.
\end{enumerate}
For any such class $\mathcal{G}_p$ of $4k$-intercyclic bipartite graphs, the permanental polynomial is not any more useful than the characteristic polynomial in distinguishing them.

\vspace{24pt}
\noindent {\bf\Large Acknowledgments}\vspace*{6pt}\\
This work is supported by the Department of Science and Technology (Govt. of India) through project DST/04/2019/002676. The second author acknowledges the support from ANRF, SRG/2022/002219. A preliminary version of this article was presented as a poster by the second author at the European Conference on Combinatorics, Graph Theory, and Applications (EuroComb 2023) held in Prague. Part of this work was done while the last author was at IIT Indore. We would like to thank the anonymous reviewers for their constructive feedback, which greatly improved the presentation of this article. One of these comments brought to our attention a result by Galluccio and Loebl (Theorem \ref{thm:gallu_loeb}). We also thank Noga Alon and George Manoussakis for helpful email conversation on listing cycles.


\bigskip
{\bf \large Contact Information}\\

\begin{tabular}{lcl}
Ravindra B. Bapat	&& Indian Statistical Institute,\\
rbb@isid.ac.in && New Delhi 110016, India.\\
	&& \includegraphics[height=10pt]{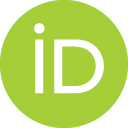}  \url{https://orcid.org/0000-0002-1127-2912}\\
				&& \\
Ranveer Singh	&& Indian Institute of Technology Indore,\\
ranveer@iiti.ac.in && Indore 453552, India. \\
	&& \includegraphics[height=10pt]{ORCIDiD_icon128x128.png}  \url{https://orcid.org/0000-0002-4566-0331}\\
					&& \\
Hitesh Wankhede	&& The Institute of Mathematical Sciences (HBNI),\\
hiteshwankhede9@gmail.com &&  Chennai 600113, India.\\&&\includegraphics[height=10pt]{ORCIDiD_icon128x128.png}  \url{https://orcid.org/0000-0003-3431-6305}\\
\end{tabular}

\end{document}